\documentclass[a4paper,12pt]{article}
\usepackage{amsmath, amsthm, amssymb}
\usepackage{latexsym}
\usepackage{amsfonts}
\usepackage{graphicx}
\begin{document}

\newtheorem{theorem}{Theorem}[section]
\newtheorem{example}{Example}
\newtheorem{observation}{Observation}
\newtheorem{exercise}{Exercise}
\newtheorem{remark}{Remark}
\newtheorem{fact}{Fact}
\newtheorem{proc}{Procedure}
\newtheorem{conj}{Conjecture}
\newtheorem{lemma}[theorem]{Lemma}
\newtheorem{prop}[theorem]{Proposition}
\newtheorem{corr}[theorem]{Corollary}
\newtheorem*{abst}{Abstract}
\newtheorem{definition}{Definition}
\newtheorem{claim}{Claim}
\newtheorem{algo}{Algorithm}

\title{Equal Sum Sequences and Imbalance Sets of Tournaments}

\author{Muhammad Ali Khan\\
 \small{Department of Mathematics and Statistics}\\
 \small{University of Calgary, 2500 University Drive NW Calgary, AB Canada T2N 1N4}\\
 \small{muhammkh@ucalgary.ca}}

\date{}
 \maketitle

\begin{abstract}
Reid conjectured that any finite set of non-negative integers is the score set of some tournament and Yao gave a non-constructive proof of Reid's conjecture using arithmetic arguments. No constructive proof has been found since. In this paper, we investigate a related problem, namely, which sets of integers are imbalance sets of tournaments. We completely solve the tournament imbalance set problem (TIS) and also estimate the minimal order of a tournament realizing an imbalance set. Our proofs are constructive and provide a pseudo-polynomial time algorithm to realize any imbalance set. Along the way, we generalize the well-known equal sum subsets problem (ESS) to define the equal sum sequences problem (ESSeq) and show it to be NP-complete. We then prove that ESSeq reduces to TIS and so, due to the pseudopolynomial time complexity, TIS is weakly NP-complete.   
\\

\noindent \small{\textit{Keywords and phrases}: Tournament, score set, imbalance sequence, imbalance set, Ried's theorem, partial tournament, pseudo-polynomial time, equal sum sequences problem, equal sum subsets problem. 
\\

}
\normalsize
\end{abstract}

\section{Introduction}
A tournament is an orientation of a complete simple graph. In a tournament, the \emph{score} $s_{i}$ of a vertex $v_{i}$ is the number of arcs directed away from that vertex, that is, the outdegree of $v_{i}$. The \emph{score sequence} of a tournament is formed by listing the scores in nondecreasing order. 
Let us write $[x_{i}]_{1}^{n}$ to denote a sequence with $n$ terms. Landau \cite{landau1} gave a simple characterization of the score sequences of tournaments. 

\begin{theorem}\label{landau}
A sequence $[s_{i}]_{1}^{n}$ of non-negative integers in nondecreasing order is the score sequence of a tournament if and only if for every $I\subseteq\{1,2,\ldots,n\}$, 
\begin{equation}\label{land}
\sum_{i\in I}{s_{i}}\geq{\left|I\right|\choose 2},
\end{equation}
\noindent with equality when $\left|I\right|=n$, where $\left|I\right|$ is the cardinality of the set $I$.
\end{theorem}

Several proofs of Landau's theorem have appeared over the years \cite{bang1, brualdi1, griggs1, landau1, thomassen1} and it continues to play a central role in the theory of tournaments and their generalizations. Brualdi and Shen \cite{brualdi2} strengthened Landau's theorem by deriving a set of inequalities that are individually stronger than inequalities (\ref{land}) but are collectively equivalent to these inequalities. 

The set of scores of vertices in a tournament is called the \textit{score set} of the tournament. Reid \cite{reid1} conjectured that any finite nonempty set $S$ of non-negative integers is the score set of some tournament. He gave a constructive proof of the conjecture for the cases $\left|S\right|=1,2,3$, while Hager \cite{hager1} settled the cases $\left|S\right|=4,5$. In 1986, Yao announced a nonconstructive proof of Reid's theorem by arithmetic arguments \cite{yao1}. Pirzada and Naikoo \cite{pirzada1} obtained the construction of a tournament with a given score set in the special case when the score increments are increasing. However, so far no constructive proof has been found for Reid's theorem in general. 

In a digraph, the \emph{imbalance} of a vertex $v_{i}$ is defined as $t_{i} = d_{i}^{+} -d_{i}^{-}$, where $d_{i}^{+}$ and $d_{i}^{-}$ are respectively the outdegree and indegree of $v_{i}$. The \emph{imbalace sequence} of a digraph is formed by listing the vertex imbalances in nonincreasing order. If $T$ is a tournament with imbalance sequence $[t_{i}]_{1}^{n}$, we say that $T$ \emph{realizes} $[t_{i}]_{1}^{n}$. Mubayi, Will and West \cite{west1} gave necessary and sufficient conditions for a sequence of integers to be the imbalance sequence of a simple digraph. 

\begin{theorem}\label{west} 
A sequence of integers $[t_{i}]_{1}^{n}$  with $t_{1}\geq\cdots\geq t_{n}$ is an imbalance sequence of a simple digraph if and only if $\sum_{i=1}^{j}t_i\leq j(n-j)$, for $1\leq j\leq n$ with equality when $j=n$.
\end{theorem}

On rearranging the imbalances in nondecreasing order, we obtain the equivalent inequalities $\sum_{i=1}^{j}t_i\geq j(j-n)$, for $1\leq j\leq n$ with equality when $j=n$.

Koh and Ree \cite{koh1} showed that if an additional parity condition is satisfied the sequence $[t_{i}]_{1}^{n}$ can be realized by a tournament. In fact, they proved the result in the more general setting of hypertournaments. The following corollary of Theorem 6 in \cite{koh1} provides a characterization of imbalance sequences of tournaments. 

\begin{theorem}\label{seq}
A nonincreasing sequence $[t_{i}]_{1}^{n}$ of integers is the imbalance sequence of a tournament if and only if $n-1, t_{1}, \ldots, t_{n}$ have the same parity and 
\begin{equation}\label{charac}
\sum_{i=1}^{j}{t_{i}} \leq j(n-j), 
\end{equation}
\noindent for $j = 1, \ldots , n$ with equality when $j = n$.
\end{theorem}

In a digraph, the set of imbalances of the vertices is called its \textit{imbalance set} \cite{pirzada2}. In \cite{pirzada2} the following result regarding the imbalance sets of oriented graphs is proved. 

\begin{theorem}\label{pir}
Let $P = \{p_1 , \ldots, p_m \}$ and $Q = \{-q_1 , \ldots, -q_n \}$, where $p_1 < \cdots < p_m$ and $q_1 < \cdots < q_n$ are positive integers. Then there exists an oriented graph with imbalance set $P \cup Q$.
\end{theorem}

Due to the interest in Reid's score set theorem, it is natural to ask if a similar result holds for imbalacne sets of tournaments. Furthermore, since a constructive proof of Reid's theorem has not yet been found, it would be interesting to look for an algorithm that generates a tournament from its imbalance set. In this paper we address both questions. We study the following decision problem and its search version. 

\begin{definition}[\textbf{Tournament Imbalance Set Problem (TIS)}]\label{TIS}  
Given a set $Z$ of integers, decide if $Z$ is the imbalance set of a tournament. 
\end{definition}

In Section \ref{odd case}, we first show that the obvious necessary conditions for the existence of tournament imbalance sets are not sufficient. We then completely characterize the sets of odd integers that are imbalance sets of tounraments. In Section \ref{even case}, we treat the case of even integers, which is more involved. We show that any set of even integers that contains at least one positive and at least one negative integer or only consists of a single element 0, is the imbalnce set of a partial tournament in which each vertex is joined to every other vertex except one. However, not all such sets are imbalance sets of tournaments. This is followed by necessary and sufficient conditions for a set of even integers to be a tournament imbalance set. In Section \ref{algorithm}, we define a new variant of the \textit{equal sum subsets problem} (ESS) called \textit{equal sum sequences problem} (ESSeq). We show that ESSeq is NP-hard and ESSeq reduces to TIS in polynomial time. Furthermore, we propose a pseudo-polynomial time algorithm that determines if a set of integers is a tournament imbalance set and, in addition, generates a tournmanet realizing any such set. Thus TIS is shown to be weakly NP-complete. We also consider extremal cases and determine upper bounds for the minimal order of a tournament realizing an imbalance set.

\section{Characterizing odd imbalance sets}
\label{odd case}

Consider a tournament of order (number of vertices) $n$. Let $v_{i}$ be a vertex with score $s_{i}$ and imbalance $t_{i}$ then $t_{i}=d_{i}^{+}-d_{i}^{-}=s_{i}-(n-1-s_{i})=2s_{i}-(n-1)$, or by rearranging $s_{i} = \frac{n-1+t_{i}}{2}$. Converesly, assume that $v_{i}$ is a vertex of a tournament with $n$ vertices and $t_{i}$ is the imbalance of $v_{i}$. Let $s_{i} = \frac{n-1+t_{i}}{2}$ then $s_{i}$ is the score of $v_{i}$. Thus we have 

\begin{lemma}\label{aux}
Let $t_{i}$ be the imbalance of a vertex $v_{i}$ in a tournament. Then $s_{i}$ is the score of $v_{i}$ if and only if  
\begin{equation}\label{inter1}
s_{i} = \frac{n-1+t_{i}}{2},
\end{equation}
\noindent where $n$ is the order of the tournament. 
\end{lemma}

A tournament is said to be \emph{regular} if all the vertices have the same score \cite{chartrand1}. Clearly, there exists a regular tournament on $n$ vertices with score $s_{i}$ if and only if $n$ is odd and $s_{i}=\frac{n-1}{2}$. Therefore, the imbalance of any vertex $v_{i}$ of a regular tournament is $t_{i}=2s_{i}-(n-1)=0$ and the imbalance set of any regular tournament is $\{0\}$.   

The following is a set of obvious necessary conditions for an imbalance set of a tournament. 

\begin{theorem}\label{necessary}
If a finite nonempty set $Z$ of integers is the imbalance set of a tournament of order $n$ then all the elements of $Z$ have the same parity as $n-1$ and it either contains at least one positive and at least one negative integer or contains only a single element 0. 
\end{theorem}
\begin{proof}
If $Z$ is the imbalance set of a tournament with $n$ vertices then by Theorem \ref{seq}, the elements of $Z$ must have the same parity as $n-1$. Furthermore, either the tournament is regular and $Z=\{0\}$ or $Z$ must contain at least one positive and at least one negative integer so the corresponding imbalance sequence sums to zero. 
\end{proof}

The question is whether these conditions are also sufficient. The answer is `no' as can be seen from the following example. 

\begin{example}
Let $Z=\{6, -10\}$. Then $Z$ satisfies the necessary conditions given in Theorem \ref{necessary} and it can potentially be the imbalance set of a tournament with an odd number of vertices. However, any sequence with elements chosen from $Z$ can sum to zero only if it consists of an even number of elements (e.g., $6$, $6$, $6$, $6$, $6$, $-10$, $-10$, $-10$ ). Thus by Theorem \ref{seq}, we cannot construct a tournament imbalance sequence from $Z$ and so $Z$ is not a tournament imbalance set.  
\end{example}

Although the conditions given in Theorem \ref{necessary} are not sufficient in general, they are sufficient if $Z$ consists of odd integers. We first show that  

\begin{theorem}\label{odd}
Let $X=\{x_{1}, \ldots, x_{l}\}$ and $Y=\{-y_{1},\ldots,-y_{m}\}$ be disjoint nonempty sets of odd integers, where $x_{1}>\cdots> x_{l}$ are positive odd integers and $y_{1}<\cdots< y_{m}$ are also positive odd integers. Let $L=\sum_{i=1}^{l}{x_{i}}$, $M=\sum_{i=1}^{m}{y_{i}}$ and $n=lM+mL$. Then there exists a tournament of order $n$ with imbalance set $X\cup Y$. 
\end{theorem}
\begin{proof} We observe that $n$ is even and all the elements of $X\cup Y$ have the same parity as $n-1$. Let $x^{(p)}$ denote that $x$ is appearing in $p$ consecutive terms of a sequence. We use Theorem \ref{seq} to prove that the $n$-term sequence 
\[
[t_{i}]_{1}^{n}={x_{1}}^{(M)},\ldots,{x_{l}}^{(M)},{-y_{1}}^{(L)},\ldots,{-y_{m}}^{(L)}
\]
\noindent is the imbalance sequence of a tournament arranged in nonincreasing order. First note that   
\[
\sum_{i=1}^{M}{t_{i}}= Mt_{1}=Mx_{1}\leq M((l-1)M+mL)=M(n-M),
\]
\[
\sum_{i=1}^{2M}{t_{i}}= M(t_{1}+t_{2})=M(x_{1}+x_{2})\leq 2M((l-2)M+mL)=2M(n-2M),
\]
\[\ldots \ \ \ \ \ldots \ \ \ \ \ldots \ \ \ \ \ldots \ \ \ \ \ldots
\]
\[
\sum_{i=1}^{lM}{t_{i}}= M\sum_{i=1}^{l}x_{i}=LM\leq lM(mL)=lM(n-lM),
\]
\[
\sum_{i=1}^{lM+L}{t_{i}}= LM-Ly_{1}\leq (lM+L)(m-1)L=(lM+L)(n-lM-L),
\]
\[
\sum_{i=1}^{lM+2L}{t_{i}}= LM-L(y_{1}+y_{2})\leq (lM+L)(m-2)L=(lM+2L)(n-lM-2L),
\]
\[
\ldots \ \ \ \ \ldots \ \ \ \ \ldots \ \ \ \ \ldots \ \ \ \ \ldots
\]
and
\[
\sum_{i=1}^{n}{t_{i}}=M\sum_{i=1}^{l}{x_{i}}-L\sum_{i=1}^{m}{y_{i}}=0=n(n-lM-mL).
\]

So inequality (\ref{charac}) holds for $j=M,2M,\ldots,lM,lM+L,lM+2L,\ldots,lM+mL(=n)$ with equality when $j=n$. Now suppose that for some other value $j=j_{0}$ we have $\sum_{i=1}^{j_{0}}{t_{i}} > j_{0}(n-j_{0})$ and $j_{0}$ is the smallest such integer. But then $t_{j_{0}}>n-2j_{0}+1$ and as $j_{0}\neq M,2M,\ldots,lM,lM+L,lM+2L,\ldots,n$, we have $t_{j_{0}+1}= t_{j_{0}}>n-2j_{0}+1>n-2j_{0}-1=n-2(j_{0}+1)+1$. Thus $\sum_{i=1}^{j_{0}+1}{t_{i}} > (j_{0}+1)(n-j_{0}-1)$, showing that $j_{0}+1\neq M,2M,\ldots,lM,lM+L,lM+2L,\ldots,n$. Continuing in this way leads to a contradiction as we must reach one of $M,2M,\ldots,lM,lM+L,lM+2L,\ldots,n$ in finitely many steps. 
\end{proof}

Together Theorems \ref{necessary} and \ref{odd} immediately give the following necessary and sufficient conditions for odd tournament imbalance sets. 

\begin{theorem}\label{oddsuff}
A finite nonempty set of odd integers is the imbalance set of a tournament if and only if it contains at least one positive and at least one negative integer. 
\end{theorem}

\section{The case of even imbalances}
\label{even case} 
Mubayi, Will and West \cite{west1} considered simple digraphs with maximum number of arcs that realize imbalance sequences.  

\begin{lemma}\cite{west1}\label{aux2}
Let $D$ be a simple digraph with maximum number of arcs realizing the imbalance sequence $[t_{i}]_{1}^{n}$. Then any vertex in $D$ has at most one non-neighbour and the number of arcs in $D$ equals $\sum_{i=1}^{n}\left\lfloor \frac{n-1+t_{i}}{2}\right\rfloor$.   
\end{lemma}

A \textit{partial tournament} is a simple digraph obtained by removing one or more arcs from a tournament \cite{brualdi1}. We say that a partial tournament of order $n$ is a \emph{near tournament} if each vertex is joined to all the other vertices except exactly one. Clearly, every near tournament has even order. 

In this section, we characterize the sets of even integers that are imbalance sets of tournaments. Recall that $\{0\}$ is the imbalance set of every regular tournament. Therefore, in the remainder of this section we focus on nonzero sets of even integers. Example 1 shows that not every set of even integers that satisfies the necessary conditions of Theorem \ref{necessary} is the imbalance set of a tournament. We can nevertheless prove that any such set is the imbalance set of a near tournament.  

\begin{theorem}\label{even}
Let $X=\{x_{1}, \ldots, x_{l}\}$ and $Y=\{-y_{1},\ldots,-y_{m}\}$ be disjoint nonempty sets of even integers, where $x_{1}>\cdots> x_{l}$ are non-negative even integers and $y_{1}<\cdots< y_{m}$ are positive even integers. Suppose that $X\cup Y\neq \{0\}$. Let $L=\sum_{i=1}^{l}{x_{i}}$, $M=\sum_{i=1}^{m}{y_{i}}$ and $n=lM+mL$. Then there exists a near tournament of order $n$ with imbalance set $X\cup Y$.
\end{theorem}
\begin{proof}
Since $L$ and $M$ are even, $n$ is even and so we cannot construct a tournament of order $n$ with imbalance set $X\cup Y$. By mirroring the proof of Case 1 of Theorem \ref{odd}, we can show that the $n$-term sequence 
\[
[t_{i}]_{1}^{n}={x_{1}}^{(M)},\ldots,{x_{l}}^{(M)},{-y_{1}}^{(L)},\ldots,{-y_{m}}^{(L)}
\]
\noindent is the imbalance sequence of a simple digraph. Let $D$ be a realization of $[t_{i}]_{1}^{n}$ with maximum number of arcs. Since all the imbalances $t_{i}$ are even while $n-1$ is odd, by Lemma \ref{aux2} the number of arcs in $D$ is 
\[
\sum_{i=1}^{n}\left\lfloor \frac{n-1+t_{i}}{2}\right\rfloor =\sum_{i=1}^{n}\frac{n-2+t_{i}}{2}=\frac{n(n-2)}{2}, 
\]

\noindent which is $\frac{n}{2}$ less than the number of arcs of a tournament of order $n$. Therefore, Lemma \ref{aux2} implies that every vertex of $D$ has exactly one non-neighbour and $D$ must be a near tournament. 
\end{proof}

The following result shows that under certain conditions we can transform $D$ into a tournament by adding a suitable number of vertices.

\begin{theorem}\label{evensuff}
Let $X$, $Y$, $l$, $m$, $L$, $M$ and $n$ be as defined in Theorem \ref{even}. 
\item (i) If $0\in X\cup Y$ then there exists a tournament of order $n+1$ with imbalance set $X\cup Y$. 
\item (ii) If there exists an $x_{p}\in X$ and  (not necessarily distinct) $-y_{q},-y_{r}\in Y$ such that $x_{p}=y_{q}+y_{r}$ then there exists a tournament of order $n+3$ with imbalance set $X\cup Y$. 
\item (iii) If there exists a $-y_{p}\in Y$ and (not necessarily distinct) $x_{q},x_{r}\in X$ such that $y_{p}=x_{q}+x_{r}$ then there exists a tournament of order $n+3$ with imbalance set $X\cup Y$. 
\end{theorem}

\begin{proof} Let $T$ be a near tournament realizing the imbalance sequence $[t_{i}]_{1}^{n}$ as defined in the proof of Theorem \ref{even}. For a vertex $v_{i}$ in $T$ let $v'_{i}$ denote the unique non-neighbour of $v_{i}$. Also let $(u,v)$ denote an arc directed from vertex $u$ to vertex $v$. In the three cases we can transform $T$ into a tournament as follows.   

\noindent \textit{(i)} Add a vertex $v$ to $T$ in such a way that for every pair of non-adjacent vertices $v_{i}$ and $v'_{i}$ we insert the arcs $(v_{i},v'_{i})$, $(v'_{i},v)$ and $(v,v_{i})$. Thus the imbalance of all the vertices of $T$ is preserved and the new vertex $v$ has imbalance 0. Since every vertex of $T$ has been linked with every other vertex of $T$ as well as the new vertex $v$, the resulting digraph is a tournament. 

\begin{figure}[htb]
	\centering
		\includegraphics[scale=1.3]{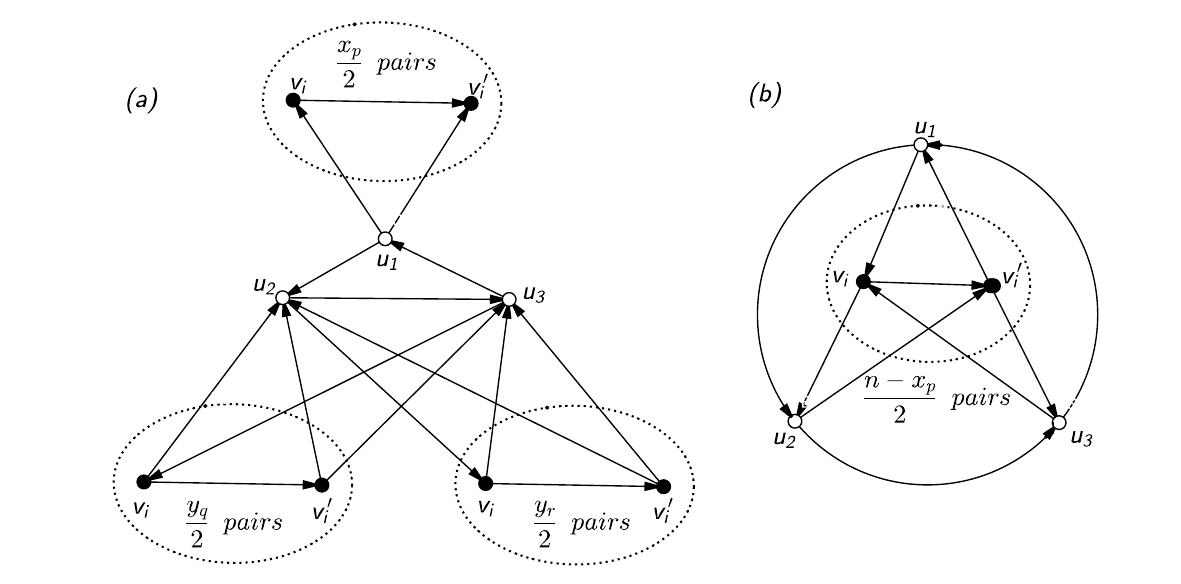}
	\caption{Construction of tournament in case (ii). Figure (a) represents step 1-2 and figure (b) step 3.}
	\label{tournamentex}
\end{figure}

\noindent \textit{(ii)} Add three new vertices $u_{1}$, $u_{2}$ and $u_{3}$ to $T$ and insert arcs in the following manner: 
\begin{enumerate}
\item Insert $(u_{1},u_{2})$, $(u_{2},u_{3})$ and $(u_{3},u_{1})$.
\item Choose any $\frac{x_{p}}{2}=\frac{y_{q}+y_{r}}{2}$ pairs $\{v_{i},v'_{i}\}$ of non-neighbouring vertices in $T$ and insert $(v_{i},v'_{i})$. Out of these choose any $\frac{y_{q}}{2}$ pairs. For each of these pairs insert the arcs $(u_{1},v_{i})$, $(u_{1},v'_{i})$, $(v_{i},u_{2})$, $(v'_{i},u_{2})$, $(v'_{i},u_{3})$ and $(u_{3},v_{i})$. For the other $\frac{y_{r}}{2}$ pairs insert the arcs $(u_{1},v_{i})$, $(u_{1},v'_{i})$, $(v_{i},u_{3})$, $(v'_{i},u_{3})$, $(v'_{i},u_{2})$ and $(u_{2},v_{i})$. 
\item For the remaining $\frac{n-x_{p}}{2}$ pairs $\{v_{i},v'_{i}\}$ of non-neighbours, insert the arcs  $(u_{1},v_{i})$, $(v'_{i},u_{1})$, $(v_{i},v'_{i})$, $(v_{i},u_{2})$, $(u_{2},v'_{i})$, $(u_{3},v_{i})$ and $(v'_{i},u_{3})$. 
\end{enumerate}
Since every vertex is joined with every other vertex, the resulting digraph is a tournament. Furthermore, the imbalance of each vertex of $T$ is preserved, while the new vertices $u_{1}$, $u_{2}$ and $u_{3}$ have imbalances $x_{p}$, $-y_{q}$ and $-y_{r}$ respectively. 

\noindent \textit{(iii)} The proof is essentially the same as that of case (ii). 
\end{proof}

Theorem \ref{evensuff} can be generalized and it is the generalized version that is of interest to us as it leads to the characterization of even imbalance sets. However, we stated and proved Theorem \ref{evensuff} to provide the reader a concrete perspective of what is happening in the more abstract setting of Theorem \ref{evensuffgen}.  

\begin{theorem}\label{evensuffgen}
Let $X$, $Y$, $l$, $m$, $L$, $M$ and $n$ be as defined in Theorem \ref{even}. The set $X\cup Y$ is the imbalance set of a tournament if any one of the following conditions is satisfied: 

\item (i) $0\in X\cup Y$,

\item(ii) there exist an odd number of (not necessarily distinct) $x_{p_1},\ldots,x_{p_{2r+1}}\in X$ and an even number of (not necessarily distinct) $-y_{q_1},\ldots,-y_{q_{2s}}\in Y$ such that $\sum_{j=1}^{2r+1}x_{p_j}=\sum_{j=1}^{2s}y_{q_j}$, 

\item(iii) there exist an odd number of (not necessarily distinct) $-y_{p_1},\ldots,-y_{p_{2r+1}}\in Y$ and an even number of (not necessarily distinct) $x_{q_1},\ldots,x_{q_{2s}}\in X$ such that $\sum_{j=1}^{2r+1}y_{p_j}=\sum_{j=1}^{2s}x_{q_j}$. 
\end{theorem}

\begin{proof}
Let $T$ be a near tournament realizing the imbalance sequence $[t_{i}]_{1}^{n}$ as defined in the proof of Theorem \ref{even}. 

\noindent \textit{(i)} The proof is exactly the same as part (i) of Theorem \ref{evensuff}. 

\noindent \textit{(ii)} Add $2r+2s+1$ new vertices labelled $u_{p_1},\ldots,u_{p_{2r+1}}, u_{q_1}, \ldots,u_{q_{2s}}$ to $T$. Note that in the construction that follows we will relabel them in different ways, such as $u_{1},\ldots,u_{2r+2s+1}$, for the sake of convenience. We insert arcs in $T$ using the following procedure. 

\begin{figure}[htb]
	\centering
		\includegraphics[scale=1.2]{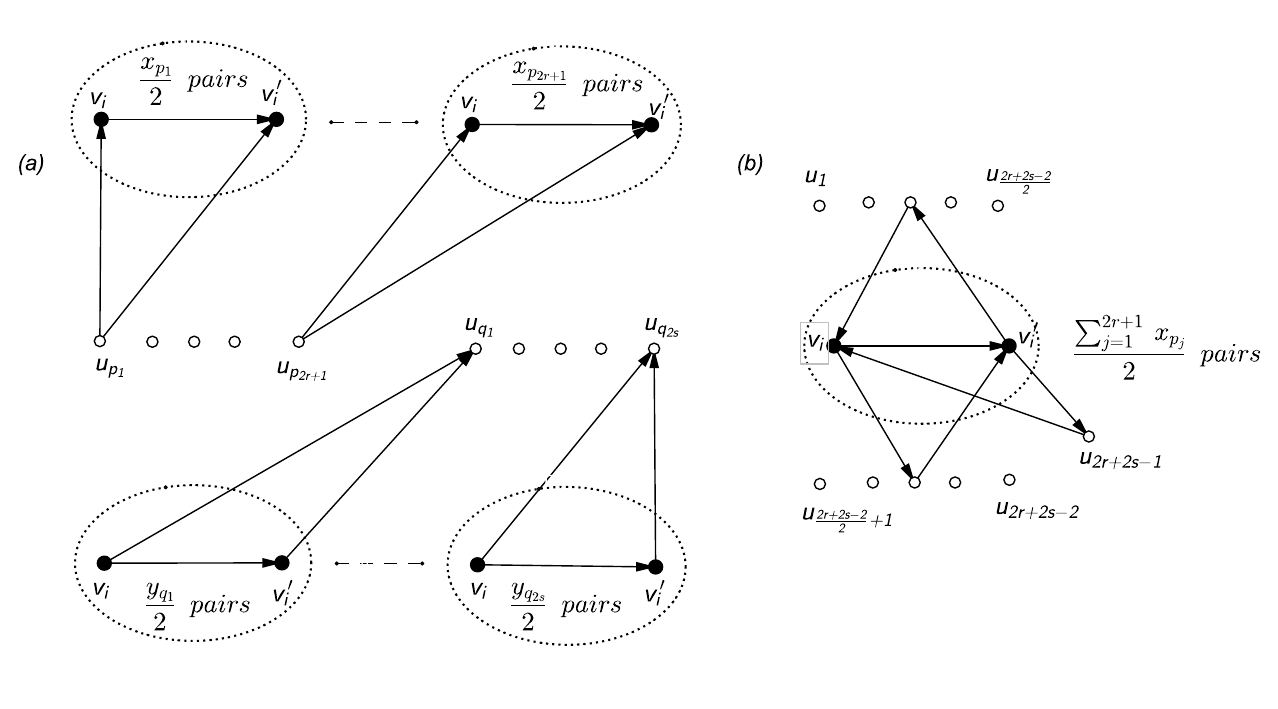}
	\caption{Figure (a) represents steps 1-4 and figure (b) step 5 of \textsc{Add Arcs}.}
	\label{tournamentex}
\end{figure}

\vspace{2mm}
\textsc{Add Arcs}:
\begin{enumerate}
\item Insert arcs so that the newly added vertices induce a regular tournament of order $2r+2s+1$. 

\item Choose any $\frac{\sum_{j=1}^{2r+1}x_{p_j}}{2}=\frac{\sum_{j=1}^{2s}y_{q_j}}{2}$ pairs $\{v_{i},v'_{i}\}$ of non-neighbouring vertices in $T$ and order them arbitrarily. For each $i=1,\ldots,\frac{\sum_{j=1}^{2r+1}x_{p_j}}{2}$ insert the arc $(v_{i}, v'_{i})$. Therefore, the imbalances of $v_{i}$ and $v'_{i}$ change by $+1$ and $-1$ respectively, for all $i$. 

\item For each $j=1,\ldots,2r+1$ choose $i=\frac{\sum_{h=1}^{p_{_{j-1}}}x_h}{2}+1,\ldots,\frac{\sum_{h=1}^{p_{j}}x_h}{2}$ and insert $x_{p_{j}}$ arcs $(u_{p_j},v_{i})$ and $(u_{p_j},v'_{i})$. This gives $u_{p_{j}}$ the imbalance $x_{p_{j}}$.  

\item For each $j=1,\ldots,2s$ choose $i=\frac{\sum_{h=1}^{q_{_{j-1}}}y_h}{2}+1,\ldots,\frac{\sum_{h=1}^{q_{j}}y_h}{2}$ and insert the arcs $(v_{i},u_{q_j})$, $(v'_{i},u_{q_j})$. Thus the imbalance of $u_{q_{j}}$ is $-y_{q_{j}}$. But the imbalances of $v_{i}$ and $v'_{i}$ are still perturbed by $+1$ and $-1$ respectively, for $i=1,\ldots,\frac{\sum_{j=1}^{2r+1}x_{p_j}}{2}$.

\item For every $i=1,\ldots,\frac{\sum_{j=1}^{2r+1}x_{p_j}}{2}$ list the $u$'s that are not already linked with $v_{i}$ and $v'_{i}$. There are exactly $2r+2s-1$ such $u$'s, for each $i$. Label them arbitrarily from $1,\ldots,2r+2s-1$. For $j=1,\ldots,\frac{2r+2s-2}{2}$ insert the arcs $(u_{j},v_{i})$ and $(v'_{i},u_{j})$. For $j=\frac{2r+2s-2}{2}+1,\ldots,2r+2s-2$ insert the arcs $(u_{j},v_{i})$ and $(v'_{i},u_{j})$. Finally, insert the arcs $(u_{2r+2s-1},v_{i})$ and $(v'_{i}, u_{2r+2s-1})$. This preserves all the imbalances.  

\item For the remaining $\frac{n-\sum_{j=1}^{2r+1}x_{p_j}}{2}$ pairs $\{v_{i}, v'_{i}\}$ of non-neighbours insert the arc $(v_{i},v'_{i})$. Label all the $u$'s arbitrarily from $1,\ldots,2r+2s+1$. For $j=1,\ldots,\frac{2r+2s}{2}$ insert the arcs $(u_{j},v_{i})$ and $(v'_{i},u_{j})$. For $j=\frac{2r+2s}{2}+1,\ldots,2r+2s$ insert the arcs $(u_{j},v_{i})$ and $(v'_{i},u_{j})$. Finally, insert the arcs $(u_{2r+2s+1},v_{i})$ and $(v'_{i}, u_{2r+2s+1})$. This preserves all the imbalances.

\begin{figure}[htb]
	\centering
		\includegraphics[scale=1.5]{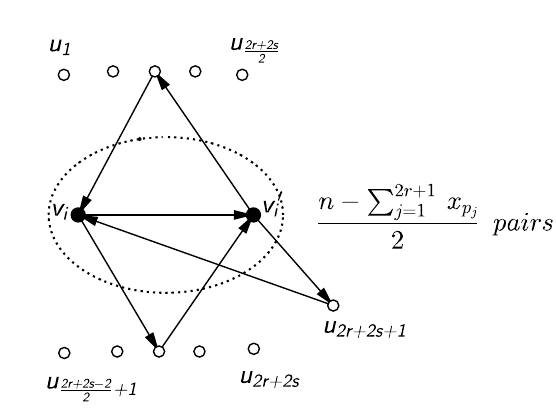}
	\caption{Inserting the remaining arcs in step 6 of \textsc{Add Arcs}.}
	\label{tournamentex}
\end{figure}

\end{enumerate}
\vspace{2mm}

Since every vertex is joined with every other vertex, the resulting digraph is a tournament. Furthermore, the imbalance of each vertex of $T$ is preserved, while the new vertices $u_{p_1},\ldots,u_{p_{2r+1}}, u_{q_1},\ldots,u_{q_{2s}}$ have imbalances $x_{p_1},\ldots,x_{p_{2r+1}}$, $-y_{q_1},\ldots,-y_{q_{2s}}$ respectively. 

\noindent \textit{(iii)} The proof is essentially the same as that of case (ii).
\end{proof}

The reader can easily draw parallels between the proofs of Theorems \ref{evensuff}(ii) and \ref{evensuffgen}(ii). For instance, the first and the last steps of both proofs are essentially achieveing the same target while steps 2-5 of the later are similar to but more complicated than step 2 of the former. 

Analyzing the above proof leads to a couple of simpler sufficient conditions for tournament imbalance sets. The first is one is a fairly straightforward consequence of Theorem \ref{evensuff} (i).   
\begin{corr}\label{zero}
If $Z$ is the empty set or it contains at least one positive and at least one negative even integer then $Z\cup \{0\}$ is the imbalance set of a tournament. 
\end{corr} 

The second condition is not as obvious and is more of an arithmetic result than a combinatorial one. First, we note that for any positive integer $p\geq 1$, the set $\{2^{p}, -2^{p}\}$ is not a tournament imbalance set as any zero sum sequence formed by the elements of this set necessarily consists of an even number of elements. However, the following sufficient condition shows that any other set of positive and negative even integers containing a power of $2$ is a tournament imbalance set.  
   
\begin{corr}\label{arithmetic}
Let $Z$ be a finite nonempty set of even integers containing at least one positive and at least one negative integer. Suppose $Z$ contains an element of the form $2^{p}$ or $-2^{p}$, for some positive integer $p\geq 1$, and $Z\neq \{2^{p}, -2^{p}\}$. Then $Z$ is the imbalance set of a tournament.
\end{corr}
\begin{proof} Let us assume that for some positive integer $p\geq 1$, $2^{p}$ is an element of $Z$. Choose any negative element $-y\in Z$. Then $y=r2^{q}$, where $q\geq 1$ is a positive integer and $r\geq 1$ is an odd positive integer such that if $r=1$ then $q\neq p$. (If this is not possible, we can start with $-2^{p}\in Z$ and choose an $x=r2^{q}$ from $Z$ with $q\neq p$.) Without loss of generality, let $p=\max\{p,q\}$. We have 
\[\underbrace{2^{p}+\cdots+2^{p}}_{r \textnormal{ terms}}=\underbrace{y+\cdots+y}_{2^{p-q} \textnormal{ terms}},
\] 
\noindent and by Theorem \ref{evensuffgen}, $Z$ is a tournament imbalance set.  
\end{proof}

After deriving a number of sufficient conditions for tournament imbalance sets of even integers, the natural question is whether the sufficient conditions given in Theorem \ref{evensuffgen} are also necessary. The answer is positive as seen from the following result.

\begin{theorem}\label{evennecess}
Let $Z=X\cup Y$ be a finite nonempty set of even integers, where $X$ is the set of non-negative integers and $Y$ is the set of negative integers in $Z$. Then $Z$ is the imbalance set of a tournament if and only if either $Z=\{0\}$ or both $X$ and $Y$ are nonempty and satisfy one of the conditions (i), (ii) or (iii) of Theorem \ref{evensuffgen}.    
\end{theorem}
\begin{proof}
The sufficiency follows from Theorem \ref{evensuffgen}. 

To prove the necessity, suppose that $0\notin X\cup Y$ and let $X\cup Y$ be the imbalance set of a tournament of order $k$. This implies that we can form a sequence $[t_{i}]_{1}^{k}$ consisting of an odd number of not necessarily distinct terms from the elements of $X\cup Y$ that sums to zero. Since $k$ is odd, therefore either the number of terms from $X$ is odd or the number of terms from $Y$ is odd, but not both. Thus we have an odd (respectively even) number of terms $x\in X$ and an even (respectively odd) number of terms $-y\in Y$ such that $\sum{x}=\sum{y}$.  
\end{proof}

\section{Algorithmic aspects}
\label{algorithm}
The aim of this section is to study the \textit{tournament imbalance set problem} (TIS) and present an algorithm that generates a tournament realizing any tournament imbalance set. We begin by proving a theorem on the lengths of equal sum sequences chosen from two set of non-negative integers that will play a crucial role in developing the algorithm. 

\begin{theorem}\label{zerosumsequence}
Let $X$, $Y$, $l$, $m$, $L$, $M$ and $n$ be as defined in Theorem \ref{even}. If $k=p+q$ is the minimum odd number such that there exists a $p$-term sequence from $X$ and a $q$-term sequence from $-Y=\{y:-y\in Y\}$ having the same sum, then $k<n$.  
\end{theorem}
\begin{proof}
We observe that $k$ equals the minimal length of a zero sum sequence from $X\cup Y$. We prove the result by induction on $n$. Note that according to the conditions of Theorem \ref{even}, $X\cup Y\neq \{0\}$ and so the minimum possible value of $n$ is 4 that only corresponds to the sets $X=\{2\}$ and $Y=\{-2\}$. Since it is not possible to form a zero sum sequence with odd number of terms from $\{2, -2\}$. The next smallest value of $n$ is $6$ that corresponds to the sets $\{2,0,-2\}$, $\{2,-4\}$ and $\{4,-2\}$. Each of these sets admits a zero sum sequence of length $k=3$ and so $k<n$.   

Now we aim to show that the result holds for any $n>6$ by assuming that it holds for all values less than $n$. Let $X$ and $Y$ be any two sets of integers corresponding to $n$ and let $k\geq 3$ be the minimum odd number such that there exists a $k$-term zero sum sequence $a_{1},\ldots, a_{p},-b_{1},\ldots,-b_{q}$, where $k=p+q$, $a_{1}>\cdots>a_{p}\in X$ and $-b_{1}>\cdots> -b_{q}\in Y$. Assume that $k>n$, then $k\geq 5$. The sequence $a_{1}+a_{2},\ldots,a_{p},-b_{1}\ldots,-b_{q-1}-b_{q}$  is a zero sum sequence of minimal odd length $k'=k-2$ corresponding to the sets $X'=X-\{a_{1},a_{2}\}\cup \{a_{1}+a_{2}\}$ and $Y'=Y-\{-b_{q-1},-b_{q}\}\cup\{-b_{q-1}-b_{q}\}$. For the sets $X'$ and $Y'$ we have $n'<n-2$ and so $k'>n'$, contradicting the induction hypothesis.    
\end{proof} 

Given a set of non-negative integers, we call the search problem of finding two disjoint nonempty subsets that have identical sums the \textit{equal sum subsets problem} (ESS). Several authors \cite{bazgan1,woeginger1} have considered the corresponding decision and optimization problems. Additionally, many variants of ESS have been studied in literature. For instance, if we require subsets to be found from two different sets of positive integers the problem is called \textit{equal sum subsets from two sets} (ESST). It is known that ESS and ESST are weakly NP-hard as they admit pseudo-polynomial time algorithms \cite{cieliebak1,woeginger1}. The best known algorithm for the ESS is the dynamic programming procedure by Bazgan, Santha and Tuza \cite{bazgan1} that runs in $O(\left|I\right|\times{Sum}^{2})$ time, and determines all possible solutions of an ESS instance. Here $\left|I\right|$ and $Sum$ respectively denote the number of elements and the sum of elements of the input set. This procedure can be easily adapted to solve ESST \cite{cieliebak1}. Here we are interested in the following variation of ESS. 

\begin{definition}[\textbf{Equal Sum Sequences Problem (ESSeq)}]\label{ESSeq}
Given two sets $X$ and $Y$ of non-negative integers and a positive integer $k$, find two nonempty finite sequences $[x]$ and $[y]$ consisting of elements from $X$ and $Y$ respectively, with each element allowed to repeat at the most $k$ times, such that $\sum x = \sum y$. 
\end{definition} 

We now study the complexity of ESSeq. Clearly, ESSeq is in class NP. Furthermore, ESST corresponds to the special case $k=1$ of ESSeq. Thus ESSeq is NP-hard. In fact, ESSeq is weakly NP-hard as we can solve any instance $ESSeq(X,Y,k)$ by using the multisets $X^{(k)}$ and $Y^{(k)}$, in which each element is repeated $k$ times, as input for the pseudo-polynomial ESST algorithm. Let us call the resulting algorithm, which finds all possible solutions to an ESSeq instance, \textsc{Equal Seq}. We have shown the following. 

\begin{theorem}\label{esseq np} 
The ESSeq decision (search) problem is weakly NP-complete (weakly NP-hard). 
\end{theorem}

On the other hand, any algorithm that solves the even case of TIS must be able to check the existence of equal sum sequences $[x]_{1}^{a}$ and $[y]_{1}^{b}$ from any given nonempty finite sets $X$ and $Y$ of even integers such that $a$ and $b$ have different parity. Let $hX$ denote the set obtained by multiplying every element of a set $X$ by $h$ and $X+\{h\}$ denote the set obtained by adding a number $h$ to every element of $X$. We can solve an instance $ESSeq(X,Y,k)$ of the ESSeq decision problem by using any TIS algorithm to solve $\left|X\right|+1$ even instances $TIS(2X,2Y)$, $TIS(2(X+\{x\}),2Y)_{x\in X}$ of TIS.  

\begin{theorem}\label{np}
The odd case of TIS can be solved in linear time. On the other hand, the even case of TIS is NP-complete. Hence in general TIS is NP-complete. 
\end{theorem}

We now present a pseudo-polynomial time algorithm that not only solves TIS but also generates a tournaments realizing any tournament imbalance set. Our algorithm is based on the proofs of Theorems \ref{odd}, \ref{evensuffgen} and \ref{zerosumsequence}. First we form a suitable $n$-term imbalance sequence and then realize it as a tournament. Lemma \ref{aux2} can be used to construct a simple digraph, with maximum possible number of arcs, realizing an imbalance sequence $[t_{i}]_{1}^{n}$. The idea is to start with an arbitrary vertex $v$ having imbalance $t_{i}$ and attach it to $\left\lfloor \frac{n-1+t_{i}}{2}\right\rfloor$ vertices by arcs directed away from $v$. If $t_{i}$ has the same parity as $n-1$ then it is joined with $n-1-\left\lfloor \frac{n-1+t_{i}}{2}\right\rfloor$ other vertices by arcs directed towards $v$. Otherwise, it is joined with $n-2-\left\lfloor \frac{n-1+t_{i}}{2}\right\rfloor$ other vertices by arcs directed towards $v$. Thus $v$ is joined to every vertex except possibly one. These steps are then repeated for every vertex without attaching any new arcs to the preprocessed vertices. We name this $O(n^{2})$ procedure \textsc{Max Realization}.  
 
Now suppose that $Z$ is a finite nonempty set of integers arranged in decreasing order. Form the sets $X=\{z\in Z:z\geq 0\}=\{x_{1}, \ldots, x_{l}\}$ and $Y=\{z\in Z:z<0\}=\{-y_{1}, \ldots, -y_{m}\}$ arranged in decreasing order. Let $L=\sum_{i=1}^{l}{x_{i}}$, $M=\sum_{i=1}^{m}{y_{i}}$ and $n=lM+mL$ as in the earlier proofs. The following algorithm outputs a tournament that realizes $Z$, whenever such a tournament exists. 

\begin{algo}[\textsc{Imbalance Set}]\label{generate}
\begin{enumerate} 

\item If either $X$ or $Y$ is empty, then $Z$ is not a tournament imbalance set. Stop.
	
\item If elements of $Z$ have different parity, $Z$ is not a tournament imbalance set. Stop.
	
\item Form the sequence $[t_{i}]_{1}^{n}={x_{1}}^{(M)},\ldots,{x_{l}}^{(M)},{-y_{1}}^{(L)},\ldots,{-y_{m}}^{(L)}.$
	
\item Call the procedure \textsc{Max Realization} to realize $[t_{i}]_{1}^{n}$ as a simple digraph $D$ with maximum number of arcs. 
	
\item If elements of $Z$ have odd parity, output $D$. End. 
	
\item If elements of $Z$ have even parity, call \textsc{Equal Seq} with the input $(X^{(n)}, (-Y)^{(n)}, n)$ to find sequences $[x]_{1}^{a}$ and $[y]_{1}^{b}$, with $a$ and $b$ having different parity and $\sum x=\sum y$. If no such sequences exist then $Z$ is not a tournament imbalance set. End. 
	
\item Add $2a+2b+1$ isolated vertices to $D$.  
	
\item Call \textsc{Add Arcs} to add $a+b$ vertices and arcs to $D$ to form a tournament $T$. Return $T$. 

\end{enumerate}
\end{algo}	

The following result shows that the procedure \textsc{Imbalance Set} runs in pseudo-polynomial time and hence TIS is weakly NP-hard. 	
\begin{theorem}\label{correct}
Algorithm \ref{generate} is correct and runs in pseudo-polynomial time. 
\end{theorem} 
\begin{proof}
The correctness follows immediately from Theorems \ref{odd}, \ref{even}, \ref{evensuffgen} and \ref{zerosumsequence}. In particular, Theorem \ref{zerosumsequence} guarantees that in the case when $Z$ is a set of even integers, step 6 of Algorithm \ref{generate} necessarily finds the required sequences if they exist. Now note that the computational complexity of Algorithm \ref{generate} is dominated by steps 4 and 6. Step 4 can be performed in $O(n^{2})=O((lM+mL)^{2})$ time, whereas step 6 takes $O(( n\left| X \right| + n\left| Y \right|)\times(n\sum_{x\in X}x + n\sum_{-y\in Y}y)^{2})=O(n^{3}(l+m)\times(L+M)^{2})=O(n^{3}\left|Z\right|\times (L+M)^{2})$ time. The overall complexity is therefore $O(\left|Z\right|\times n^{5})$, which is pseudo-polynomial since $n$ depends on the numeric value of the input.      
\end{proof}

Thus we can use Algorithm \ref{generate} to check if a given set of integers is the imbalance set of a tournament and moreover, to construct a tournament realizing the set, if it exists. We now illustrate Algorithm \ref{generate} by showing how it generates a tournament realizing the imbalance set $\{4,2,-2\}$. 

\begin{figure}[htb]
\centering
\includegraphics[scale=0.7]{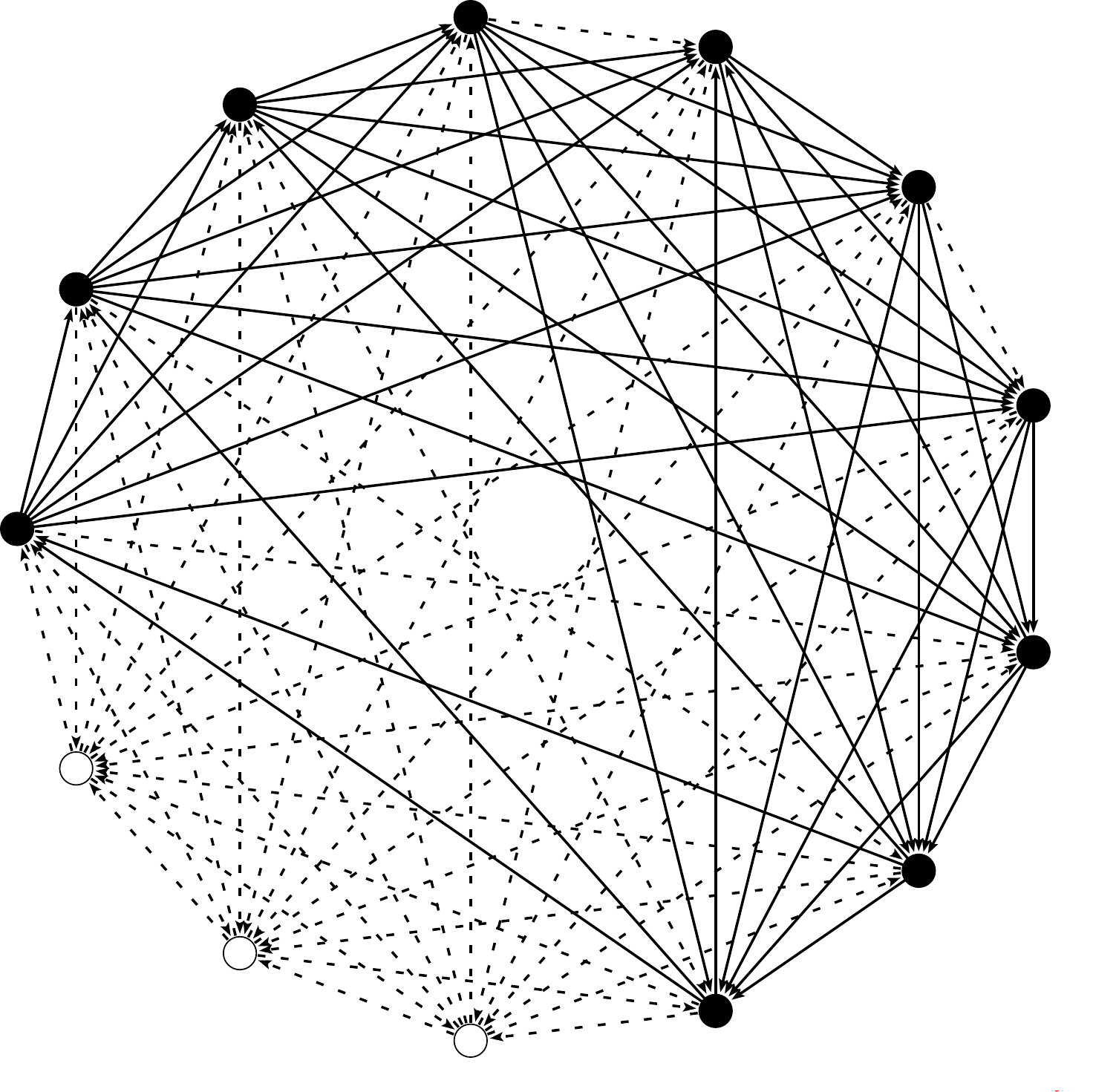}
\caption{A tournament realizing the imbalance set $\{4, 2, -2\}$ obtained from  Algorithm 1.}
\label{tournamentex}
\end{figure}

\begin{example}
Consider the set $Z=\{4, 2, -2\}$. Since $Z$ satisfies the conditions in the first two steps of Algorithm 1, the algorithm goes to step 3 and forms the sequence $4, 4, 2, 2, -2, -2, -2, -2, -2, -2$. Step 4 calls the procedure \textsc{Max Realization} to output a simple digraph of order 10 realizing $Z$. However, this simple digraph is only a near tournament and not a tournament (see the digraph induced by the black vertices in Figure \ref{tournamentex}). Since the elements of $Z$ have even parity, the algorithm proceeds to step 6 and finds the the sequences $[x_{i}]_{1}^{1}=4$ and $[y_{i}]_{1}^{2}=-2,-2$ with odd and even number of terms respectively, such that $\sum{x_{i}} = -\sum{y_{i}}$. Then step 7 adds 3 new vertices (colored white in Figure \ref{tournamentex}) to the near tournament obtained in step 4. In the end, step 8 adds the arcs (dashed arcs in Figure \ref{tournamentex}) necessary to form a tournament of order 13 in such a way that the imbalances of the old vertices are preserved and the new vertices have imbalances $4$, $-2$, and $-2$. The output of Algorithm 1 is a tournament with imbalance sequence $4, 4, 4, 2, 2, -2, -2, -2, -2, -2, -2, -2, -2$ as shown in Figure \ref{tournamentex}. 
\end{example}

The results presented in Sections 2, 3 and 4 can be used to estimate the order of a tournament realizing an imbalance set. Let us denote by $ord(Z)$ the minimal order of a tournament realizing an imbalance set $Z$. 

\begin{theorem}\label{extremal}
Let $Z$ be a tournament imbalance set (i.e., it satisfies the conditions of Theorem 2.4 or Theorem 3.6). Define $X$, $Y$, $l$, $m$, $L$, $M$ and $n$ as in the earlier results. 
\item(i) If $Z$ consists of odd integers then $ord(Z)\leq n=lM+mL$. 
\item(ii) If $Z$ consists of even integers and $0\in Z$ then $ord(Z)\leq n+1=lM+mL+1$.
\item(iii) If $Z$ consists of even integers and $0\notin Z$ then $ord(Z)< 2n=2(lM+mL)$.   
\end{theorem} 
\begin{proof}
The proof of (i) and (ii) follows from Theorem \ref{odd}, while (iii) follows from Theorem \ref{evensuff}(i). For (iv), observe that the order of the tournament constructed in the proof of Theorem \ref{evensuffgen} is $n+2r+2s+1$. From Theorem \ref{zerosumsequence}, $2r+2s+1< n$. As a result, the constructed tournament can be at the most of order $2n-1$. 
\end{proof}

\bigskip
\noindent \textbf{Acknowledgements}

\noindent The author would like to thank Professor K\'{a}roly Bezdek for having many useful discussions on the topic and helping in impproving this manuscript. The author is also thankful to Dr. Shariefuddin Pirzada for drawing his attention to the imbalance set problem. This research was performed at the Center for Computational and Discrete Geometry, Department of Mathematics and Statistics, University of Calgary.

\end{document}